\documentclass[a4paper,12pt]{article}
\usepackage{amssymb,amsthm,amsmath,amsfonts}
\usepackage{graphicx}
\usepackage{cite}
\usepackage{color}

 \newtheorem{thm}{Theorem}[section]
 \newtheorem{cor}[thm]{Corollary}
 
 \newtheorem{prop}[thm]{Proposition}
 \theoremstyle{definition}
 \newtheorem{defn}[thm]{Definition}
 \theoremstyle{remark}
 \newtheorem{rem}[thm]{Remark}
 
 \numberwithin{equation}{section}

\begin{document}

\title{On some quaternionic generalized slice re\-gular  functions}

\small{
\author
{Jos\'e Oscar Gonz\'alez-Cervantes$^{\footnote{corresponding author}}$}
\vskip 1truecm
\date{\small $^{*}$ Departamento de Matem\'aticas, ESFM-Instituto Polit\'ecnico Nacional. 07338, Ciudad M\'exico, M\'exico\\ Email: jogc200678@gmail.com}}

\maketitle

{\small
\begin{abstract} 
The quaternionic valued functions of a quaternionic variable, often referred to as slice regular functions   has been   studied extensively due to the large number of generali\-zed results of the theory of one complex variable,   see \cite{cgs,CSS,GSC,GS2,gssbook,gp,gpr,GS}  and the references given there.
Recently,  several global properties of  these functions has been found of the study of a   differential operator, see  
\cite{GlobalOp,GP_2, G, GG1,GG2}.
 Particularly, given a structural set $\psi$ the  Borel-Pompieu formula induced by the operator 
${}^{\psi}G$ and its consequences in  
  the slice regular function theory were studied in  \cite{GG1}.
\\	
The aim of this paper  is to present some global and local properties of a kind of quaternionic generalized slice regular functions. We shall see that the global properties are consequences of the   study  of the  perturbed  global-type   operator:
 \begin{align*}   {}^{\psi}G_v [f]  :=  {}^{\psi} G  [f]  -\frac{{\bf x}_{\psi}}{ 2} ({\bf x}_{\psi}  v + v {\bf x}_{\psi} ) f    , \end{align*}  
 where $v$ is a quaternionic constant and $f$ is a quaternionic-valued continuously differentiable function with domain in $\mathbb H$ since our 
  generalized slice regular function space coincides with  
  $\textrm{Ker} {}^{\psi_{\textrm{st}}}G_v$ associated to an axially symmetric s-domain, where   the  $\psi_{\textrm{st}}$
	is standard structural set.
	\\	
	Among the local properties  studied  in this work are the versions of   Splitting Lemma and Representation Theorem  
	that show us a deep relationship between this  generalized slice regular function space with   a  complex generalized analytic    function space  on each slice.
	\end{abstract}
}
 
{\bf Subjclass}: {Primary 30G35}

\maketitle
\section{Introduction}

The theory quaternionic   slice regular functions    has been   studied extensively in the last 15 years,   see \cite{cgs,CSS,GSC,GS2,gssbook,gp,gpr,GS}  and the references given there.

The  study  of operator 
\begin{align*}    G   := \displaystyle  \|  {\bf x}    \|^2 \psi_0\partial_0   +  {\bf x}   \sum_{k=1}^3 x_k \partial_k .   \end{align*} has given us some  global properties of    the slice regular functions  such as Cauchy-Riemann equations, Borel-Pompieu, Cauchy formulas and  Cauchy integral theorem and an analog of the chain role, see   \cite{GlobalOp,GP_2, G, GG1,GG2}. This operator can be associated to any structural set $\psi $ as follows: 
\begin{align*}   {}^{\psi}G   := \displaystyle  \|  {\bf x}_{\psi}   \|^2 \psi_0\partial_0   +  {\bf x}_{\psi}   \sum_{k=1}^3 x_k \partial_k ,  \end{align*}
see \cite{GG1}.

The aim of this paper is to present the  Borel-Pompieu  and Cauchy integral formulas  induced by  
 \begin{align*}   {}^{\psi}G_v [f]  :=  {}^{\psi} G  [f]  -\frac{{\bf x}_{\psi}}{ 2} ({\bf x}_{\psi}  v + v {\bf x}_{\psi} )  f  , \end{align*}  
 where $v$ is a quaternionic constant and $f$ is a continuously differentiable function. For the standard structural set $\psi_{\textrm{st}}$ we shall see    
   a conformal covariant property of ${}^{\psi_{\textrm{st}}} G_v $ and the consequences of all these facts  in the function space $\textrm{Ker}{}^{\psi_{\textrm{st}}}G_v$ since we shall see that if  $ {}^{\psi_{\textrm{st}}}G_v$ is associated to an  axially symmetric slice domain  $\Omega\subset \mathbb H$ then   	$\textrm{Ker}{}^{\psi_{\textrm{st}}}G_v$			if formed by  continuously differential functions  $f:\Omega\to \mathbb H$ such that  
$$ \dfrac{\partial f\mid_{\Omega_{\bf i}}}{ \partial \bar z_{\bf i}} + \frac{1}{ 4} (  v - {\bf i} v {\bf i} ) f \mid_{\Omega_{\bf i}} = 0, \quad \textrm{on }  \  \Omega_{\bf i}  .$$
for all ${\bf i}\in \mathbb S^2$, which are a kind of      quaternionic generalized slice regular   functions. What is more,    these functions   are given in terms  of  pairs of 
 complex  generalized analytic function associated to  a  complex Vekua-type problem.

In addition, this writing  presents some  local formulas  of these  kind of quaternionic generalized slice regular functions such as  Splitting Lemma, Representation theorem, Cauchy's  formula,  Identity principle,  Liouville's theorem  and  Morera's theorem.

After this brief introduction,  the structure of the paper is the following:  Some preliminaries to the quaternionic slice regular function theory and 	to the   operator ${}^{\psi}G$  are given  in Section 2.  Our principal results are stated and proved in Section 3 which  is divided in two subsection, the first one studies the function theory induced by the perturbed global-type operator ${}^{\psi}G_v$, while the second one presents the global and local formulas of a kind of quaternionic generalized   slice regular  functions. Finally, 
Section 4 presents the conclusions of this paper and future works.

\section{Preliminaries}
 
The skew-field of quaternions $\mathbb H$ is formed by   
 $ q=\sum_{k=0}^3q_ke_k$  where $q_k\in\mathbb R$  for  all $k$  and the quaternionic imaginary units satisfy: 
$e_0=1$, {}  $e_1^2=e_2^2= e_3^2 =-1$,  {}  $e_1e_2=e_3$, $e_2e_3=e_1 $, and  $e_3e_1=e_2$. The real part of $q\in \mathbb H$ is   $q_0$  and  the  vector part of  $q$ is   ${\bf q} = q_1e_1 +  q_2e_2  + q_3e_3$  which is   usually identified  with    $ ( q_1, q_2, q_3)\in\mathbb R^3$.  The quaternionic conjugation and the quaternionic modulus of  $q$ are $ \bar q  =q_0 - {\bf  q} $ and  $\|q\|=\sqrt{ q\bar q }= \sqrt{q_0^2+ q_1^2+  q_2^2 + q_3^2 }$, respectively. 

On the other hand, given $q,r\in \mathbb H$ the inner product  is given by 
$$\langle q,r\rangle = \sum_{k=0}^3 q_k r_k = \frac{1}{2}\left( \bar q r + \bar r q  \right) = \frac{1}{2}\left(  q \bar r + r \bar q  \right).$$

The unit open ball in $\mathbb H$ is 
$\mathbb B^4(0,1):\{ q \in \mathbb H \ \mid \  \|q \| < 1\}$ and  the unit sphere 
$ \mathbb S^3: = \{     q\in \mathbb  H \ \mid \ \| q\|=1 \}$. 
The unit sphere  in $\mathbb R^3$ is denoted by  $ \mathbb S^2 := \{    {\bf q} \in \mathbb R^3 \ \mid \ \|  {\bf q} \|=1 \}$.

A  quadruple of quaternions  $\psi=\{\psi_0, \psi_1,\psi_2,\psi_3\}$ is called a structural set if 
$\langle  \psi_k; \psi_m \rangle  =\delta_{k,m} $ 
for  $k, m\in \{ 0,1,2,3\}$. Therefore $\psi$ is a basis of $\mathbb H$ as real-linear space and then each quaternion is written by  $  x_{\psi} = \sum_{k=0}^3 x_k\psi_k$ with  $x_k\in \mathbb R$ for all $k$ and note that $x_{\psi_{\textrm{st}}} =x$.   
  Denote  
$  \overline{x_{\psi}} = \sum_{k=0}^3 x_k \overline{\psi_k}$ and given $x,y \in \mathbb H$ set  
$ \langle x,y \rangle_{\psi}=\sum_{k=0}^3{x_k }y_k $, 
where $  x_{\psi} = \sum_{k=0}^3 x_k\psi_k$ and $  x_{\psi} = \sum_{k=0}^3 y_k\psi_k$. The orthonormality property of $\psi$ allows to see that 
$$ \langle x,y \rangle_{\psi}=\frac{1}{2}\left( \overline{x_{\psi} }y_{\psi} + \overline{y_{\psi}} x_{\psi}  \right) =\frac{1}{2}\left(x_{\psi}\overline{y_{\psi}} + y_{\psi} \overline{x_{\psi}}   \right).$$

 The 4-tuple   $\psi_{\textrm{st}}:=\{e_0,e_1, e_2, e_3\}$ is  called the  canonical,  or standard, structural set of $\mathbb H$. Then one sees that  
$\langle x,y \rangle_{\psi_{\textrm{st}}} = \langle x, y \rangle$ for all $x,y\in \mathbb H$, see, e.g.  \cite{shapiro1} and  \cite{sudbery}.

\subsection{On quaternionic slice regular functions }

Let's recall concepts and properties of the theory of quaternionic slice regular functions  for achieving our goal.

Given   ${\bf i} \in \mathbb S^2$ one sees that  ${\bf i}^2=-1$ and so   $\mathbb C({\bf i}):=\{x+{\bf i}y  \mid  x,y\in\mathbb R\} \cong \mathbb C$ as fields.

Let $\Omega\subset \mathbb H$ be a domain, i.e., $\Omega $ is an open and connected set, then $\Omega$  is called axially symmetric  if  
      $\Omega\cap \mathbb R\neq \emptyset$  and if  $x+{\bf i}y \in \Omega$, where $x,y\in\mathbb R$ and
 ${\bf i}\in \mathbb S^2$   then  
 $\{x+{\bf j}y \ \mid  \ {\bf j}\in\mathbb{S}^2\}\subset \Omega$. Moreover,   
   $\Omega$ is      called   slice domain, or s-domain, 
if   $\Omega_{\bf i} = \Omega \cap \mathbb C({\bf i})$ is  a domain in   $\mathbb C({\bf i})$ 
 for all ${\bf i}\in\mathbb S^2$.

 Let $\Omega\subset\mathbb H$ be an axially symmetric s-domain open set.   A  real differentiable 
function $f:\Omega\to \mathbb{H}$ is called 
 quaternionic left slice regular function, or slice regular function,   
if   
\begin{align*}
 \dfrac{\partial  f\mid_{\Omega_{\bf i}}
 }{ \partial \bar z_{\bf i}} 
:=
\frac{1}{2}\left (\frac{\partial}{\partial x}+{\bf i} \frac{\partial}{\partial y}\right )
f\mid_{\Omega_{\bf i}}=0, \  \textrm{ on  $\Omega_{\bf i}= \Omega\cap \mathbb C({\bf i})$,}
\end{align*}
for all ${\bf i}\in \mathbb{S}^2$       and its   derivative, or Cullen's derivative see \cite{GSC}, is    
$f'=\displaystyle 
 {\partial}_{{\bf i}}f\mid_{_{\Omega\cap \mathbb C({\bf i})}} = \frac{\partial}{\partial x} f\mid_{_{\Omega\cap \mathbb C({\bf i})}}$. 
By $\mathcal{SR}(\Omega)$ we mean  the quaternionic right-linear space of the  slice regular functions defined on  $\Omega$,  see       \cite{newadvances, CGS3,  CSS}.

     Splitting Lemma and Representation Theorem are two important  results in 
the slice regular function theory.

Let $\Omega \subset\mathbb{H}$ be an axially symmetric s-domain and    $f\in\mathcal{SR}(\Omega)$. 
\begin{enumerate}
\item (Splitting Lemma) Given      ${\bf i},{\bf j}\in \mathbb{S}$  orthogonal to each other  
there exist $F,G\in Hol(\Omega_{\bf i})$,   holomorphic functions from  $\Omega_{\bf i}$ to $\mathbb{C}({{\bf i}})$, 
 such that  $f_{\mid_{\Omega_{\bf i}}} =F +G  {\bf j}$ on $\Omega_{\bf i}$, 
see \cite{CSS}. 
\item (Representation Formula) Given   $q=x+{\bf i}_q y \in \Omega$,    {where} $x,y\in\mathbb R$ and  
${\bf i}_q \in \mathbb S^2$,     then   
\begin{align*}   
f(x+{\bf i}_q y) = \frac {1}{2}(1-{\bf i}_q {\bf i} )   f(x+{\bf i}y) + \frac {1}{2}(1+{\bf i}_q {\bf i} )   f(x-{\bf i}y),
\end{align*} 
for all ${\bf i}\in \mathbb S^2$, see  \cite{newadvances}.
\end{enumerate}
Moreover, given  ${\bf i},{\bf j}\in \mathbb{S}$,  orthogonal to each other, from the previous results one has the following  operators:  
	    $
     Q_{ {\bf i},{\bf j}} :    \mathcal{SR}(\Omega) \to  Hol(\Omega_{  {\bf i}})+ Hol(\Omega_{ {\bf i}}){  {\bf j} }$ 
and 		$ P_{  {\bf i},{\bf j} } :   Hol(\Omega_{  {\bf i}})+ Hol(\Omega_{  {\bf i}}){  {\bf j}} \to  \mathcal{SR}(\Omega)$
		defined
		by  
$ Q_{ {\bf i} , {\bf j} } [f ]= f\mid_{\Omega_{{\bf i}}} = f_1+f_2{\bf j}$ for all $f\in\mathcal{SR}(\Omega)$, where
 $f_1,f_2\in Hol(\Omega_{\bf i})$, 
and 
\begin{align*}
   P_{ {\bf i},{\bf j} }[g](x+ {\bf i}_q y)= \frac{1}{2}\left[(1+ {\bf i}_q{ \bf  i})g(x-y{{\bf i}}) + (1- {\bf i}_q {  {\bf i}}) g(x+y {  {\bf i} })\right]
	, \quad \forall x+{\bf i}_qy \in \Omega,
	\end{align*} for all  $g\in  Hol(\Omega_{  {\bf i} })+ Hol(\Omega_{  {\bf i} }){  {\bf j} }$.
			
			{What is more, these operators  satisfy that }
 \begin{align*} P_{  {\bf i} ,  {\bf j}  }\circ Q_{ {\bf i} ,{\bf j}  }= \mathcal I_{\mathcal{SR}(\Omega)} \quad \textrm{and} \quad    Q_{  {\bf i} ,  
{\bf j} }\circ P_{ {\bf i} ,  {\bf j}   }= \mathcal I_{ Hol(\Omega_{ {\bf i} })+ Hol(\Omega_{ {\bf i} }){ {\bf j} } },
\end{align*} 
where $\mathcal I_{\mathcal{SR}(\Omega)}$  and $\mathcal I_{ Hol(\Omega_{ {\bf i} })+ Hol(\Omega_{ {\bf i} }){  {\bf j}  } }$ are  the identity
 operators in  $\mathcal{SR}(\Omega)$ and    in  ${ Hol(\Omega_{ {\bf i} })+ Hol(\Omega_{ {\bf i} }){ {\bf j} } }$ respectively, 
see  \cite{GS}.

The development in quaternionic power series.
		 Given  $f,g\in\mathcal {SR}(\mathbb B^4(0,1))$ there exist sequences of quaternions  $(a_n)_{n=0}^\infty$ and  $(b_n)_{n=0}^\infty$  such that  
$ 
f(q)= \sum_{n=0}^{\infty} q^n a_n$ and $  g(q)= \sum_{n=0}^{\infty} q^n b_n$  for all $q\in\mathbb B^4(0,1)$ and 
their $*$-product   is     defined by    
$ f*g(q)= \sum_{n=0}^{\infty} q^n  \sum_{k=0}^{n} a_k b_{n-k}$, 
 for all $q\in\mathbb B^4(0,1)$, see \cite{newadvances}.

Let  $\Omega\subset \mathbb H$ be an axially symmetric s-domain.  Identity Principle for the  slice regular      functions shows that if 
$f\in \mathcal{SR}(\Omega)$ and if  there exists ${\bf i}\in\mathbb S^2$ such that    $Z_f \cap \Omega_{\bf i} =
 \{q \in \Omega \ \mid \  f (q) = 0\} \cap \Omega_{\bf i}$   has
an accumulation point then $f \equiv 0$ on $\Omega$, see \cite{newadvances}. 
 Cauchy's  Integral Formula   for slice regular functions    shows   
		that  given $f\in \mathcal SR(\Omega)$,  $a\in\mathbb R$ and $r>0$ such that 
\begin{align*}
q\in 	\overline{\Delta_q(a,r)}= \{ x+ y{\bf i} \ \mid (x-a)^2+y^2 \leq r^2   \} \subset \Omega.
	\end{align*}
	  one has that		      
	\begin{align}   \label{CauchyForm_SR} 
f(q) = 
\frac{1}{2\pi }\int_{\partial\Delta_q(a,r) } (\zeta- q)^{-1} d\zeta_{{\bf i}_q}  f (\zeta), 
\end{align}
	where  $d\zeta_{{\bf i}_q} := - d\zeta {{\bf i}_q}$ and   from    Cauchy's   Integral Theorem one has that    
\begin{align}\label{CauchyTheo_SR} \int_{\Gamma} f d\zeta =0, \end{align}
 for all ${\bf i}\in \mathbb S^2$ and for any closed,    { homotopic
to a point and } piecewise $C^1$       curve   $\Gamma\subset \Omega_{\bf i}$ , see \cite{CSS}.

 Identity Principle for the  slice regular   functions.  
If $f\in \mathcal{SR}(\Omega)$ and there exists ${\bf i}\in\mathbb S^2$ such that   $Z_f \cap \Omega_{\bf i} =
 \{q \in \Omega \ \mid \  f (q) = 0\} \cap \Omega_{\bf i}$   has
an accumulation point then $f \equiv 0$ on $\Omega$, see \cite{newadvances}.

   Liouville's Theorem for quaternionic slice regular functions. If  $f\in\mathcal {SR}(\mathbb H)$ is a  bounded function   then $f$ is a quaternionic  constant function,see \cite{CSS}.

Finally,   Morera's Theorem shows that if  $h:\Omega\to \mathbb H$ satisfies   
that  \begin{align*} \int_{\Gamma}h d\zeta =0,
\end{align*}
for any closed,     homotopic
to a point and   piecewise $C^1$ curve   $\Gamma\subset \Omega_{\bf i}$ 
  and for all ${\bf i}\in \mathbb S^2$ then  $h\in \mathcal{SR}(\Omega)$, see     \cite{CSS}.

\subsection{On Global operator ${}^{\psi}G$}

From now on, given an structural set  $\psi$   we shall suppose that   $\psi_0=\pm 1$  because 
 the real part of the Cauchy-Riemann  operator $ \bar{\partial}  _{{\bf i}}$ is    $\displaystyle \frac{1}{2}\frac{\partial}{\partial x_0}$ for all ${\bf i}\in \mathbb S^2$.    The operators  ${}^\psi G$  and  ${}^\psi G_r$ 
were  introduced in \cite{G}  as follows:	
\begin{align*}  {}^\psi G[f] &  := \displaystyle  \|  \vec x_{\psi}  \|^2\psi_0 \partial_0 f +  \vec x_{\psi}  \sum_{k=1}^3 x_k \partial_k f ,\\
 {}^\psi G_r[f]  & :=\|  \vec x_{\psi}  \|^2 \psi_0 \partial_0 f +  \sum_{k=1}^3 x_k (\partial_k f ) \vec x_{\psi} ,\end{align*}
for any continuously differentiable $\mathbb H$-valued function $f$.  It is easily verified that    
$\textrm{Ker}({}^{\psi}G)$ is a    quaternionic  right-linear space and $\textrm{Ker}({}^\psi G_r)$  is a     quaternionic  left-linear space. 	For short denote $G={}^{\psi_{\textrm{st}}}G$.

The   Borel-Pompieu-  and the Cauchy-type formulas induced by $ {}^\psi G$ and $ {}^\psi G_r$ were proved in  \cite{GG1}.
 
 The global  quaternionic Borel-Pompeiu-type formula. 
	Let $\Omega\subset\mathbb H$ be a domain such that  $\partial \Omega$ is a 3-dimensional compact smooth surface and $\overline{\Omega}\subset \mathbb H\setminus \mathbb R$. Then
		\begin{align}   & \displaystyle \int_{\partial \Omega}  \|\vec{\tau}_{\psi}\|^2  \big[  {g(\tau)} \nu_\tau^\psi  A_{\psi}( x, \tau)  - A_{\psi}(\tau, x)  \nu_\tau^\psi  f(\tau) \big]  \nonumber \\
		& +  		\int_{\Omega} \big[ B_{\psi}(y, x)f(y)  -     
		  g(y) C_{\psi}(x,y)	\big]
		dy \nonumber   \\ 
			&  +  2	\int_{\Omega}  \big[ 	   A_{\psi}(y, x) {}^\psi G[f](y)    -   {}^\psi G_r[g](y) A_{\psi}(x,y)  \big]
		dy \nonumber \\
		= &   \label{ecua5}   \left\{ \begin{array}{ll} f(x) + g(x), &  x\in \Omega, \\ 0 ,&  x\in \mathbb H\setminus\overline{\Omega},  \end{array}       \right.    \end{align}
	for all  $f,g \in C^{1}(\overline{\Omega},\mathbb H)$, where 
the  reproducing functions are given by 
 \begin{align*}
{ A}_{\psi}( x,  t  ) = & \frac{    1}{2\pi^2} 
    \frac{   {\bf x}_{\psi}\overline{( t  _{\psi} - x_{\psi})}  {\bf  t  }_{\psi}  }{\|{\bf  x}_{\psi}\|^2\| t  _{\psi} - x_{\psi}\|^4  \| { \bf   t  }_{\psi}\|^2  }  , \\
  B_{\psi}( t   , x )= &  
	  \frac{1}{\pi^2   }  \frac{   {\bf  x}_{\psi}}{   \|{\bf  x}_{\psi}\|^2 } \big[     \frac{  { t  }_{\psi} + 3 \overline{x_{\psi}}   - 4 \overline{{ t  }_{\psi}}      }{   \| t  _{\psi}-x_{\psi}\|^4  }    \\ 
		 & + \frac{  4 (\overline{  t  _{\psi}- x_{\psi} }) [   (x_0 -  t  _{\psi} ){\bf   t  }_{\psi}  
		-  \left\langle  {\bf  t  } ,  {\bf x} \right\rangle_{\psi}  ] }{   \| t  _{\psi} - x_{\psi}\|^6}   \big],\\
  		C_{\psi}(x, t  )  = &    \frac{1}{\pi^2   }  \big[     \frac{   { t  }_{\psi} + 3 \overline{x_{\psi}}   - 4 \overline{{ t  }_{\psi}}      }{   \|x_{\psi}-  t  _{\psi}\|^4  }    \\ 
		 & + \frac{  4  [   (x_0-  t  _{\psi} ){\bf  t  }_{\psi}  -  \left\langle  {\bf  t  },  {\bf x}  \right\rangle_{\psi}  ]
		(   \overline{      t  _{\psi}  - x_{\psi}  }) }{   \|  t  _{\psi}- x_{\psi} \|^6}   \big]  \frac{  
			{\bf x}_{\psi}}{   \|  {\bf x}_{\psi}\|^2 }   .			\end{align*}
			 where    $\tau_{\psi}, x_{\psi} \in \mathbb H\setminus\mathbb R$ such that  $x_{\psi}= x_0 \psi_0 + \sum_{n=1}^3x_k \psi_k =  x_0 \psi_0 + {\bf x}_{\psi}$ and   $x_{\psi}\neq \tau_{\psi}$.
For abbreviation  denote  $A=A_{\psi_{st}}$,  $B=B_{\psi_{st}}$ and $C=C_{\psi_{st}}$. 
The differential form 
	\begin{align*}\displaystyle  \nu^\psi_x= & 2\psi_0  d\hat{x}_0+ 2  \frac{{ \bf x}_\psi}{ \|{\bf  x} _{\psi} \|^2 } \sum_{k=1}^3 x_k d \hat{x}_k  =   \sigma^\psi_x  - \frac{{\bf  x}_\psi}{\|{\bf  x}_\psi\| }\sigma_x^\psi \frac{{\bf  x}_{\psi}}{\|{\bf x}_{\psi}\| },
	\end{align*} 
 where  $\sigma^\psi_x$ is  the quaternionic differential form of the 3 dimensional volume in $\mathbb R^4$, see  \cite{GG1}.

   {The following notation is used all over the paper: 
	\begin{itemize}
		\item  {Given two domains $\Xi,\Omega\subset  \mathbb R^4$, let
		$\alpha:\Xi \to \Omega$ be a one-to-one correspondence;  if $f$
		belongs to a function space on $\Omega$ then  denote  $W_\alpha: f
		\mapsto f\circ
		\alpha$.}
		\item  {If $\beta$ is a $\mathbb
		H$-valued function then  denote   $  \  {{}^\beta M}: f \mapsto  \beta f,\quad M^\beta : f
		\mapsto   f \beta$ on the same function space}
\end{itemize}

 Particularly,  if $\Omega\subset\mathbb H$ is an axially symmetric slice domain  in \cite{GG2}
 was proved that  $\textrm{Ker}( G)\cap C^{1}(\Omega,\mathbb H) = \mathcal{SR}(\Omega) $, where $G={}^{\psi_{\textrm{st}}}G$.

Given the quatermionic M\"obius transformation  
\begin{align}\label{MoebiuspreserASSDom} 
{\mathcal T}(x)= (ax+b)(cx+d)^{-1} ,
\end{align}  
where $a,b,c,d\in\mathbb H$ satisfy   {that}  $a\bar c, (b-ac^{-1} d)\bar c, d \overline{(b-ac^{-1}d)} \in \mathbb R$ then  the conformal covariant property of $G$: 
\begin{align}\label{confcovFormG}
G\circ ({}^{A_{\mathcal T}}M\circ W_{\mathcal T})= ({}^{B_{\mathcal T}}M\circ W_{\mathcal T})\circ G
\end{align}
 on  $C^1({\mathcal T}^{-1}(\Omega),\mathbb H)$ where 
$
 A_{\mathcal T}(x)=  \bar c$, $B_{\mathcal T}(y) =     \|c\|  \|b-ac^{-1}d\| \bar c  (y - ac^{-1})^{-2}$  
 and  $y = {\mathcal T}(x)$ was  given in  \cite{GG2}.

\section{Main results}

\subsection{On a quaternionic perturbed global-type operator}

Let $\psi$ be a structural set consider  
 \begin{align*}  {}^{\psi} G_v [f]  :=  & {}^{\psi}G  [f]  -\frac{{\bf x}_{\psi}}{ 2} ({\bf x}_\psi  v + v {\bf x}_\psi ) f ,\\
{}^{\psi} G_{r,v} [f]  := &  {}^{\psi}G _r [f]  - f  \frac{{\bf x}_\psi}{ 2} ({\bf x}_\psi  v + v {\bf x}_\psi ) ,
   \end{align*}
 where $v$ is a quaternionic constant and $f$ is a continuously differentiable function on a domain in $\mathbb H$.

\begin{thm}\label{quaternionic_Borel_Pompeiu_type_form}(Quaternionic Borel-Pompeiu-type formula induced by  ${}^{\psi} G_v$ and ${}^{\psi} G_{r,v}$)
	Let $\Omega\subset\mathbb H$ be a domain such that  $\partial \Omega$ is a 3-dimensional compact smooth surface and $\overline{\Omega}\subset \mathbb H\setminus \mathbb R$. Then
	\begin{align*}   & \displaystyle \int_{\partial \Omega}  \| {\bf \tau}_{\psi}\|^2  \big[  {g(\tau)} \nu_\tau^\psi  \mathcal A_{\psi}( x, \tau,u)  - \mathcal A_{\psi}(\tau, x,v)  \nu_\tau^\psi   f(\tau) \big]  \nonumber \\
		& +  		\int_{\Omega} \big[ \mathcal B_{\psi}(y, x, v)  f(y)  -     
		   g(y) \mathcal C_{\psi}(x,y,u)	\big]
		dy \nonumber   \\ 
			&  +  2	\int_{\Omega}  \big[ 	  \mathcal  A_{\psi}(y, x ,v)  {}^\psi G_v[  f](y)    -   {}^\psi G_{r,u}[  g](y)  \mathcal A_{\psi}(x,y,u)  \big]
		dy \nonumber \\
		= &     \left\{ \begin{array}{ll} f(x) +  g(x), &  x\in \Omega, \\ 0 ,&  x\in \mathbb H\setminus\overline{\Omega},  \end{array}       \right.    \end{align*}
	for all  $f,g \in C^{1}(\overline{\Omega},\mathbb H)$, where 
 \begin{align*}
{\mathcal A}_{\psi}( x,  t  , u) = & \frac{    1}{2\pi^2} 
    \frac{ e^{\langle t  -x, u \rangle_{\psi}  }  {\bf x}_{\psi}\overline{( t  _{\psi} - x_{\psi})}  {\bf  t  }_{\psi}  }{\|{\bf  x}_{\psi}\|^2\| t  _{\psi} - x_{\psi}\|^4  \| { \bf   t  }_{\psi}\|^2  }  , \\
\mathcal 	 B_{\psi}( t   , x,v)= &  
	  \frac{1}{\pi^2   }  \frac{ e^{\langle t  -x, v \rangle_{\psi}  } {\bf  x}_{\psi}}{   \|{\bf  x}_{\psi}\|^2 } \big[     \frac{  { t  }_{\psi} + 3 \overline{x_{\psi}}   - 4 \overline{{ t  }_{\psi}}      }{   \| t  _{\psi}-x_{\psi}\|^4  }    \\ 
		 & + \frac{  4 (\overline{  t  _{\psi}- x_{\psi} }) [   (x_0 -  t  _{\psi} ){\bf   t  }_{\psi}  
		-  \left\langle  {\bf  t  } ,  {\bf x} \right\rangle_{\psi}  ] }{   \| t  _{\psi} - x_{\psi}\|^6}   \big],\\
\mathcal 		C_{\psi}(x, t  , u)  = &    \frac{1}{\pi^2   }  \big[     \frac{   { t  }_{\psi} + 3 \overline{x_{\psi}}   - 4 \overline{{ t  }_{\psi}}      }{   \|x_{\psi}-  t  _{\psi}\|^4  }    \\ 
		 & + \frac{  4  [   (x_0-  t  _{\psi} ){\bf  t  }_{\psi}  -  \left\langle  {\bf  t  },  {\bf x}  \right\rangle_{\psi}  ]
		(   \overline{      t  _{\psi}  - x_{\psi}  }) }{   \|  t  _{\psi}- x_{\psi} \|^6}   \big]  \frac{  e^{\langle t  -x, u \rangle_{\psi}  }   {\bf x}_{\psi}}{   \|  {\bf x}_{\psi}\|^2 }   .			\end{align*}
 \end{thm}

\begin{proof}
Apply formula  \eqref{ecua5}  to functions  $e^{\langle \cdot, v \rangle_{\psi}} f$ and $e^{\langle  \cdot, u\rangle_{\psi}} g$ to get  the follo\-wing:
	\begin{align*}   & \displaystyle \int_{\partial \Omega}  \|\vec{\tau}_{\psi}\|^2  \big[ e^{\langle \tau, u\rangle_{\psi}  } {g(\tau)} \nu_\tau^\psi  A_{\psi}( x, \tau)  - A_{\psi}(\tau, x)  \nu_\tau^\psi e^{\langle \tau, v\rangle_{\psi}  }  f(\tau) \big]  \nonumber \\
		& +  		\int_{\Omega} \big[ B_{\psi}(y, x)e^{\langle y, v\rangle_{\psi}  } f(y)  -     
		  e^{\langle y, u\rangle_{\psi}  } g(y) C_{\psi}(x,y)	\big]
		dy \nonumber   \\ 
			&  +  2	\int_{\Omega}  \big[ 	   A_{\psi}(y, x) {}^\psi G[ e^{\langle y, v\rangle_{\psi}  } f](y)    -   {}^\psi G_r[ e^{\langle y, u \rangle_{\psi}  } g](y) A_{\psi}(x,y)  \big]
		dy \nonumber \\
		= &    \left\{ \begin{array}{ll} e^{\langle x, v \rangle_{\psi}  }f(x) + e^{\langle x,u\rangle_{\psi}  }g(x), &  x\in \Omega, \\ 0 ,&  x\in \mathbb H\setminus\overline{\Omega},  \end{array}       \right.    \end{align*}
	for all  $f,g \in C^{1}(\overline{\Omega},\mathbb H)$.	
	From direct computations one sees that
	\begin{align}\label{G-Exp}
	{}^{\psi}G[e^{\langle x, v  \rangle_{\psi}} f](x) = & e^{\langle x,v \rangle_{\psi}} \left(  {}^{\psi}G  [f](x) -\frac{{\bf x}_{\psi}}{ 2} ({\bf x}_\psi  v + v {\bf x}_\psi ) f(x) \right) \nonumber \\ 
	= &  e^{\langle x, v \rangle_{\psi}} {}^{\psi}G_v[f](x), \nonumber  \\
{}^{\psi}G_r[e^{\langle x,v  \rangle_{\psi}} f](x) =  & e^{\langle x,v \rangle_{\psi}} \left(  {}^{\psi}G_r  [f](x) -f(x)\frac{{\bf x}_{\psi}}{ 2} ({\bf x}_\psi  v + v {\bf x}_\psi )  \right) \nonumber  \\ 
	= &  e^{\langle x,v \rangle_{\psi}} {}^{\psi}G_{r,v}[f](x),
	\end{align}
	where we denote $v=v_{\psi}$ since it is a quaternionic constant. Therefore, 
	\begin{align*}   & \displaystyle \int_{\partial \Omega}  \|\vec{\tau}_{\psi}\|^2  \big[  {g(\tau)} \nu_\tau^\psi  e^{\langle \tau , u\rangle_{\psi}  } A_{\psi}( x, \tau)  -  e^{\langle \tau , v \rangle_{\psi}  }A_{\psi}(\tau, x)  \nu_\tau^\psi   f(\tau) \big]  \nonumber \\
		& +  		\int_{\Omega} \big[ e^{\langle y  , v\rangle_{\psi}  } B_{\psi}(y, x) f(y)  -     
		   g(y) e^{\langle y , u \rangle_{\psi}  }C_{\psi}(x,y)	\big]
		dy \nonumber   \\ 
			&  +  2	\int_{\Omega}  \big[ 	  e^{\langle y  , v \rangle_{\psi}  } A_{\psi}(y, x) {}^\psi G_v[  f](y)    -   {}^\psi G_{r,u}[  g](y) e^{\langle y ,  u \rangle_{\psi}  } A_{\psi}(x,y)  \big]
		dy \nonumber \\
		= &       \left\{ \begin{array}{ll} e^{\langle  x , v \rangle_{\psi}  } f(x) +  e^{\langle  x , v \rangle_{\psi}  } g(x), &  x\in \Omega, \\ 0 ,&  x\in \mathbb H\setminus\overline{\Omega},  \end{array}       \right.    \end{align*}
	for all  $f,g \in C^{1}(\overline{\Omega},\mathbb H)$.	Then  consider $g=0$ and 
	multiply the obtained identity by    the factor $e^{\langle -x , v \rangle_{\psi}} $ on both sides. Repeat  the reasoning for   $f=0$
	 and $e^{\langle -x , u \rangle_{\psi}} $ to obtain two expression.   Finally,  add these two  expressions and note that
${\mathcal A}_{\psi}( x,  t  , u) =  e^{\langle  t  -x, u \rangle_{\psi}  } A _{\psi}( x,  t )$,  
$\mathcal 	 B_{\psi}( t   , x,v)=     e^{\langle  t  -x ,v \rangle_{\psi}  }B_{\psi}( t   , x ) $ and 
 $\mathcal 		C_{\psi}(x, t  , u)  =   e^{\langle  t  -x , u \rangle_{\psi}  } 	C_{\psi}(x, t  ) $. 
\end{proof}

\begin{cor}\label{quaternionic_Cauchy_form}(Quaternionic Cauchy-type formula induced by  ${}^{\psi} G_v$ and ${}^{\psi} G_{r,v}$)
	Let $\Omega\subset\mathbb H$ be a domain such that  $\partial \Omega$ is a 3-dimensional compact smooth surface and $\overline{\Omega}\subset \mathbb H\setminus \mathbb R$. Given $f,g \in C^{1}(\overline{\Omega},\mathbb H)$ such that 
$f \in \textrm{Ker} {}^\psi G_v$ and  $g \in \textrm{Ker} {}^\psi G_u$ then 
	\begin{align*}   & \displaystyle \int_{\partial \Omega}  \|\vec{\tau}_{\psi}\|^2  \big[  {g(\tau)} \nu_\tau^\psi  \mathcal A_{\psi}( x, \tau,u)  - \mathcal A_{\psi}(\tau, x,v)  \nu_\tau^\psi   f(\tau) \big]  \nonumber \\
		& +  		\int_{\Omega} \big[ \mathcal B_{\psi}(y, x, v)  f(y)  -     
		   g(y) \mathcal C_{\psi}(x,y,u)	\big]
		dy \nonumber   \\ 
		= &       \left\{ \begin{array}{ll} f(x) +  g(x), &  x\in \Omega, \\ 0 ,&  x\in \mathbb H\setminus\overline{\Omega},  
		\end{array}       \right.    \end{align*}
	for all  $f,g \in C^{1}(\overline{\Omega},\mathbb H)$.
\end{cor}
\begin{proof}
It's follows from Proposition \ref{quaternionic_Borel_Pompeiu_type_form}.
\end{proof}

We shall see a conformal covariant property of  $G_{v}$. 
 
\begin{prop}\label{ConfCovProp_type_G} Conformal covariant-type  property of $G_u$. Let $\Omega\subset\mathbb H$
be a domain and  $u,v\in\mathbb H$. Given the quaternionic M\"obius transformation $\mathcal T:\mathcal T^{-1}(\Omega)\to \Omega$ given by \eqref{MoebiuspreserASSDom}. 
 Denote  $y = {\mathcal T}(x)$.  Then  
\begin{align*}
  G_{u}\circ  {}^{e^{<\cdot, v> }A_{\mathcal T}}M\circ W_{\mathcal T}   =  
			{}^{E_T}M\circ    W_{\mathcal T} \circ      G_{\Gamma_{\mathcal T} }       
\end{align*}
 on   $  C^1(  \Omega ,\mathbb H)$, where 
$ G_{\Gamma_{\mathcal T}  }  : =  G      - {}^{\Gamma_{\mathcal T} }M $,
 \begin{align*}\Gamma_{\mathcal T} (y)=  & 
\frac{ \|\ell\|(y - ac^{-1})^{2} c} {2 \|c\|^3   } \left[  \left(   \ell {\bf y}c-  {\bf m}  \right)^{-1}- {\bf p} \right]   \left[  \ (  \   ( \ell {\bf y}c-  {\bf m}  ) ^{-1}- {\bf p} \  ) (u+v)  \right. \\
 & \  \  \left.  + (u+v) ( \  (\ell {\bf y}c-  {\bf m}  )^{-1}- {\bf p} ) \ \right]    \bar c  ,\\
 E_{\mathcal T} (y) = &  \frac{\|c\| }{ \|\ell\|}  e^{<T^{-1}(y), v> }  \bar c  (y - ac^{-1})^{-2} ,\end{align*}
where $\ell= (b-ac^{-1}d)^{-1}$, $m =\ell a$ and   $ p= c^{-1}d $. 
 \end{prop}
\begin{proof}
It is important to comment that  the differentiation variable of  $G_{u}$ is $x$ and the variable of $G$ is $y$. 
Given $f\in    C^1(  \Omega  ,\mathbb H)$ from \eqref{G-Exp} and \eqref{confcovFormG} one obtains  that
\begin{align*}
& G_{u}[ e^{<x, v> } {}^{A_{\mathcal T}}M\circ W_{\mathcal T} [f] (x)]\\
=&     G[  e^{<x, v> } {}^{A_{\mathcal T}}M\circ W_{\mathcal T}[f](x)]
  -\frac{{\bf x}}{ 2}({\bf x}u + u {\bf x} )  e^{<x, v> } {}^{A_{\mathcal T}}M\circ W_{\mathcal T}[f](x)  \\
 = &    e^{<x, v> }  \left\{ G[  {}^{A_{\mathcal T}}M\circ W_{\mathcal T}[f](x)] -  \frac{{\bf x}}{ 2}({\bf x}v + v {\bf x} ) {}^{A_{\mathcal T}}M\circ W_{\mathcal T}[f] \right. \\
 & \left. -\frac{{\bf x}}{ 2}({\bf x}u + u {\bf x} )   {}^{A_{\mathcal T}}M\circ W_{\mathcal T}[f](x) \right\}  \\
 = &    e^{<x, v> }  \left\{ G[  {}^{A_{\mathcal T}}M\circ W_{\mathcal T}[f]] -  \frac{{\bf x}}{ 2}({\bf x}(u+v) + (u+v) {\bf x} ) {}^{A_{\mathcal T}}M\circ W_{\mathcal T}[f] \right\} \\
= &    e^{<x, v> }  \left\{ ({}^{B_{\mathcal T}}M\circ W_{\mathcal T})\circ G[f] -  \frac{{\bf x}}{ 2}({\bf x}(u+v) + (u+v) {\bf x} ) {}^{A_{\mathcal T}}M\circ W_{\mathcal T}[f] \right\} .
\end{align*}
Due to the properties of the coefficients of $\mathcal T$ one sees that the mapping $y \to \ell yc$ preserves $\mathbb R^3$ and so 
$   {\bf   x} =  [\ell {\bf y}c-  {\bf m}  ]^{-1}- {\bf p} $. Therefore 
\begin{align*}
& G_{u}[ e^{<x, v> } {}^{A_{\mathcal T}}M\circ W_{\mathcal T} [f] (x)] \\
= &    e^{<x, v> }  \left\{ ({}^{B_{\mathcal T}}M\circ W_{\mathcal T})\circ G[f] (y) - \frac{1}{2}\left[       [\ell {\bf y}c-  {\bf m}  ]^{-1}- {\bf p}     \right]  \right. 
 \\  
&    \left[   (  [\ell {\bf y}c-  {\bf m}  ]^{-1}- {\bf p})(u+v) + (u+v) ( [\ell {\bf y}c-  {\bf m}  ]^{-1}- {\bf p}) \right] \\
 &\left.    {}^{A_{\mathcal T}}M\circ W_{\mathcal T}[f (y)] \right\} 
\end{align*}
From definitions of $A_{\mathcal T}$ and   $B_{\mathcal T}$  one has that 
\begin{align*}
& G_{u}[ e^{<x, v> } {}^{A_{\mathcal T}}M\circ W_{\mathcal T} [f] (x)] \\
= &    e^{<x, v> }  \left\{   \frac{\|c\| }{ \|\ell\|} \bar c  (y - ac^{-1})^{-2}      W_{\mathcal T} \circ G[f]  \right. 
 \\ 
&   - ( \frac{1}{2} [\ell {\bf y}c-  {\bf m}  ]^{-1}- {\bf p}  ) \left[    ( [\ell {\bf y}c-  {\bf m}  ]^{-1}- {\bf p})(u+v) \right. \\
 &\left.  \left.+ (u+v) ( [\ell {\bf y}c-  {\bf m}  ]^{-1}- {\bf p})  \right]    \bar c   W_{\mathcal T}[f] \right\} 
\end{align*}
Therefore,
\begin{align*}
& G_{u}[ e^{<x, v> } {}^{A_{\mathcal T}}M\circ W_{\mathcal T} [f] (x)] \\
= &    e^{<x, v> }  \frac{\|c\|}{  \|\ell\|} \bar c  (y - ac^{-1})^{-2}   \left\{       W_{\mathcal T} \circ G[f] - \frac{ \|\ell\|(y - ac^{-1})^{2} c} {2\|c\|^3  } \right. 
 \\ 
&     \left[   (  \ \ell {\bf y}c-  {\bf m}  \ )^{-1}- {\bf p}
 \right]\left[   ([\ell {\bf y}c-  {\bf m}  ]^{-1}- {\bf p})(u+v) \right.  \\
 &\left.\left. + (u+v) ( [\ell {\bf y}c-  {\bf m}  ]^{-1}- {\bf p})  \ \right]    \bar c   W_{\mathcal T}[f] \right\} 
\end{align*}
and according to definitions of  $\Gamma_{\mathcal T} $ and $E_{\mathcal T} $  one concludes the main identity.

\end{proof}

\begin{rem}
The previous proposition   explains the behavior of the new perturbed operators  $G+ {}^{\Gamma_{\mathcal T}}M$ in terms of our  perturbed operators.

The  conformal covariant-type  property of $G_u$ is a quaternionic version of the  chain rule in complex analysis and it  implies  that  
 $f \in \textrm{Ker} G_{\Gamma_{\mathcal T}}\cap C^{1}(\Omega,\mathbb H)$  if and only if   ${}^{e^{<\cdot, v> }A_{\mathcal T}}M\circ W_{\mathcal T} [f] \in \textrm{Ker} G_u
 C^{1}(\mathcal T^{-1}(\Omega),\mathbb H)$. What is more, the  previous   property  allows us to establish a bijective  quaternionic-right linear mapping from  $\textrm{Ker} G_{\Gamma_{\mathcal T}}\cap C^{1}(\Omega,\mathbb H)$ to $  \textrm{Ker} G_u
 \cap C^{1}(\mathcal T^{-1}(\Omega),\mathbb H)$. 

Choosing $v=-u$ one obtains  that 
 $  G_{u}\circ  {}^{e^{<\cdot, -u> }A_{\mathcal T}}M\circ W_{\mathcal T}   =  
			  e^{<\cdot, -u> }  ({}^{B_{\mathcal T}}M\circ W_{\mathcal T})\circ G  $,  
 on   $  C^1(  \Omega ,\mathbb H)$, i.e.,
 the perturbation is eliminated. 

The computations shown above can be repeated on any structural set  $\psi $ to obtain a conformal covariant-type property of ${}^{\psi}G_v$. But   our
   interest to study  some generalized  slice regular functions; that is why  we only will	consider $\psi_{\textrm{st}}$ in the following sentences.

	Note that the computations presented  in this subsection extend the results obtained  in \cite{GG1, GG2} since 
 $ {}^{\psi} G_0 =  {}^{\psi} G$ and  $ {}^{\psi} G_{r,0} =  {}^{\psi} G_r$. 	In addition, papers \cite{CG,CGK} present several results about some quaternionic  operators that can be used to study  ${}^{\psi} G_{r} $
 and ${}^{\psi} G_{r,v}$.

	\end{rem}

\subsection{On some quaternionic generalized slice regular functions}

\begin{defn}\label{GeneralizedSR}
Let $\Omega\subset \mathbb H$ be   an axially symmetric s-domain and given $v\in \mathbb H$. If $f\in C^1(\Omega,\mathbb H) $  satisfies  
$$ \dfrac{\partial f\mid_{\Omega_{\bf i}}}{ \partial \bar z_{\bf i}} +\frac{1}{ 4} (  v - {\bf i} v {\bf i} ) f\mid_{\Omega_{\bf i}}  = 0, \quad \textrm{on }  \  \Omega_{\bf i}  .$$
for all ${\bf i}\in \mathbb S^2$ then $f$ is called  $v$-slice regular  functions, that is a kind of quaternionic generalized slice regular function. The  quaternionic right-linear space formed by  these functions is denoted by 
  $\mathcal{SR}_v(\Omega)$.
\end{defn}

\begin{defn} Let $\Omega\subset \mathbb H$ be   an axially symmetric s-domain. Given $v\in \mathbb H$  and   ${\bf i}\in \mathbb S^2$ define 
$$M_{v} (\Omega_{\bf i})= \{h\in C^1(\Omega_{\bf i} , \mathbb C({\bf i}) ) \ \mid  \dfrac{\partial  h
 }{ \partial \bar z_{\bf i}} +\frac{1}{ 4} (  v -{\bf i} v {\bf i} )   h  = 0 ,\quad \textrm{on } \  \Omega_{\bf i} \}.$$
Note that if $v=v_1+v_2{\bf j}$, where $v_1,v_2\in \mathbb C({\bf i})$ with ${\bf j}\in \mathbb S^2$ orthogonal to ${\bf i}$, then 
$ v -{\bf i} v {\bf i} = 2 v_1$. Therefore  
the identity 
$$\dfrac{\partial  h
 }{ \partial \bar z_{\bf i}} +\frac{1}{ 4} (  v -{\bf i} v {\bf i} )   h  = 0 ,\quad \textrm{on } \  \Omega_{\bf i}  $$
is equivalent to 
$$\dfrac{\partial  h
 }{ \partial \bar z_{\bf i}} +\frac{1}{ 2}   v_1   h  = 0 ,\quad \textrm{on } \  \Omega_{\bf i} .$$
\end{defn}

\begin{rem}
In complex analysis it is well-known that    $M_{v} (\Omega_{\bf i})$ is a  ge\-ne\-ra\-li\-zed analytic function space    associated to a   
Vekua-type problem and several  applications in which appears the bi-dimensional Helmholtz equation, see \cite{B,V,WS}. The function space  $M_{v} (\Omega_{\bf i})$ is  considered as a generalization of the analytic   functions since it is related with  the Cauchy-Riemman operator  perturbed by the multiplication operator by the  complex number. Analogously,    $\mathcal{SR}_v(\Omega)$ is    considered a  quaternionic   generalized slice regular function space.  
\end{rem}

 Definition \ref{GeneralizedSR} and  $\textrm{Ker} G_{v}$ are related in  next result.

\begin{prop}\label{G_vSR_v}
Let $\Omega\subset \mathbb H$ be   an axially symmetric s-domain and  $v\in \mathbb H$
then  $\mathcal{SR}_v(\Omega) = \textrm{Ker} G_{v} \cap C^1(\Omega,\mathbb H)$.
\end{prop}
\begin{proof}
Set $f\in \textrm{Ker} G_{v} \cap C^1(\Omega,\mathbb H)$ or equivalently, from \eqref{G-Exp} one sees that
	$ e^{\langle \cdot,v  \rangle } f\in \textrm{Ker} G\cap C^{1}(\Omega,\mathbb H)$. From the global characterization of  the quaternionic slice regular functions  presented in Subsection 5 of \cite{GlobalOp}  one sees that    
$$ \dfrac{\partial e^{\langle \cdot, v  \rangle }f \mid_{\Omega_{\bf i}}}{ \partial \bar z_{\bf i}}=0 , \quad \textrm{on }  \  \Omega_{\bf i}  .$$
for all ${\bf i}\in \mathbb S^2$. One easily sees that  $ {\langle z, v   \rangle } = {\langle   z ,v_1 \rangle }$ for all $z\in \Omega_{\bf i} $ where $v= v_1+ v_2{\bf j}$ with  $v_1,v_2\in \mathbb C({\bf i})$ and  ${\bf j}\in \mathbb S^2$ is orthogonal to ${\bf i}$. From   direct computations one gets  that  
$$  \dfrac{\partial e^{\langle \cdot,v_1  \rangle }f \mid_{\Omega_{\bf i}}}{ \partial \bar z_{\bf i}} = e^{\langle \cdot,v_1  \rangle } \left[  \dfrac{\partial f\mid_{\Omega_{\bf i}}}{ \partial \bar z_{\bf i}} +\frac{1}{ 2} v_1 f\mid_{\Omega_{\bf i}} \right], $$ 
or equivalently
$$  \dfrac{\partial e^{\langle \cdot, v  \rangle }f \mid_{\Omega_{\bf i}}}{ \partial \bar z_{\bf i}} = e^{\langle  \cdot,v  \rangle } \left[  \dfrac{\partial f\mid_{\Omega_{\bf i}}}{ \partial \bar z_{\bf i}} +\frac{1}{ 4} (  v - {\bf i} v {\bf i} ) f\mid_{\Omega_{\bf i}} \right], $$ 
on $ \Omega_{\bf i}$ for all ${\bf i}\in \mathbb S^2$. Recall that $2v_1 =   v - {\bf i} v {\bf i}  $.  

Therefore $f\in \textrm{Ker} G_{v} \cap C^1(\Omega,\mathbb H)$  iff 
$$    \dfrac{\partial f\mid_{\Omega_{\bf i}}}{ \partial \bar z_{\bf i}} +\frac{1}{ 4} (  v - {\bf i} v {\bf i} ) f\mid_{\Omega_{\bf i}} =0, $$ 
on $ \Omega_{\bf i}$ for all ${\bf i}\in \mathbb S^2$.
\end{proof}

\begin{cor}\label{quaternionic_Cauchy_type_form}(The global  quaternionic Cauchy-type formula for $\mathcal{SR}_v(\Omega)$)
	Let $\Omega\subset\mathbb H$ be an axially symmetric s-domain  such that  $\partial \Omega$ is a 3-dimensional compact smooth surface and $\overline{\Omega}\subset \mathbb H\setminus \mathbb R$. Given $f \in \mathcal{SR}_v(\Omega)$ then 
	\begin{align*}   &- \displaystyle \int_{\partial \Omega}  \|\vec{\tau}_{\psi}\|^2     \mathcal A_{\psi}(\tau, x,v)  \nu_\tau^\psi   f(\tau)   +  		\int_{\Omega}   \mathcal B_{\psi}(y, x, v)  f(y)   	dy \nonumber   \\ 
		= & \left\{ \begin{array}{ll} f(x) , &  x\in \Omega, \\ 0 ,&  x\in \mathbb H\setminus\overline{\Omega},  \end{array}       \right.    \end{align*} 
\end{cor}
\begin{proof}Direct consequence of  Proposition \ref{G_vSR_v} and Corollary \ref{quaternionic_Cauchy_form}.
\end{proof}

The quaternionic M\"obius $\mathcal T$ given in \eqref{MoebiuspreserASSDom} preserves  axially symmetric s-domains, i.e.,  $\Omega\subset \mathbb H$  and $T^{-1}(\Omega)$ are  axially symmetric s-domains simultaneously, see \cite{GG2}.

\begin{cor}\label{SR_v_T}   Let $\Omega\subset\mathbb H$ be an  axially symmetric s-domain and set $v\in \mathbb H$. Given 
  $\mathcal T : \mathcal T^{-1}(\Omega) \to \Omega$   by \eqref{MoebiuspreserASSDom} and $f\in C^1(\Omega,\mathbb H)$.  Then  
 $  {}^{e^{\langle \cdot, v\rangle }A_{\mathcal T}}M\circ W_{\mathcal T} [f]\in \mathcal{SR}_v( {\mathcal T}^{-1} (\Omega) )$ if and only if  $  G_{\Gamma_T }[f]=0$   on $\Omega$.
 \end{cor}
\begin{proof}Its follows from  Proposition \ref{G_vSR_v}  and  Proposition \ref{ConfCovProp_type_G}.
\end{proof}
  Now 
 we shall see several local properties of $\mathcal{SR}_v(\Omega)$.

\begin{prop}\label{SGV_Split_Repre}(Splitting Lemma and Representation Theorem for $\mathcal{SR}_v(\Omega)$).
Let $\Omega \subset\mathbb{H}$ be an axially symmetric s-domain and  $v\in \mathbb H$. For   $f\in\mathcal{SR}_v(\Omega)$ one has the following: 
\begin{enumerate}
\item  Let    ${\bf i},{\bf j}\in \mathbb{S}$ be orthogonal to each other  
there exist $F, G\in  M_{v} (\Omega_{\bf i}) $      
 such that  $f_{\mid_{\Omega_{\bf i}}} =F +G  {\bf j}$ on $\Omega_{\bf i}$.

\item   Given   $q=x+{\bf i}_q y \in \Omega$,    {where} $x,y\in\mathbb R$ and  
${\bf i}_q \in \mathbb S^2$,    { one has that }   
\begin{align*}   
& f  (x+{\bf i}_q y)\\ = & \frac {1}{2}\left\{    e^{\langle ({\bf i}-{\bf i}_q)  y, v\rangle  }  (1-  {\bf i}_q {\bf i})   f  (x+{\bf i}  y) + 
e^{<({\bf i}-{\bf i}_q) y, v\rangle    } ( 1+  {\bf i}_q {\bf i})  f  (x-{\bf i}  y)   \right\} ,
\end{align*} for all ${\bf i}\in \mathbb S^2$.
\end{enumerate}
\end{prop}

\begin {proof} First of all, let us recall that $v-{\bf i}v{\bf i}\in \mathbb C({\bf i})$ for all ${\bf i}\in \mathbb S^2$.
\begin{enumerate}
\item  Set $f=f_1+f_2{\bf j}$, where $f_1,f_2\in C^1( \Omega, \mathbb C({\bf i}) )$, then 
$$ \dfrac{\partial (f_1+f_2{\bf j} )}{ \partial \bar z_{\bf i}} +\frac{1}{ 4} (  v -{\bf i} v {\bf i} )  (f_1+f_2{\bf j})   = 0 \quad \textrm{on }  \  \Omega_{\bf i}  .$$
Therefore,
$$ \dfrac{\partial f_1 }{ \partial \bar z_{\bf i}} +\frac{1}{ 4} (  v -{\bf i} v {\bf i} )  f_1  = 0,  \quad \textrm{and  }  
 \dfrac{\partial  f_2  }{ \partial \bar z_{\bf i}} +\frac{1}{ 4} (  v -{\bf i} v {\bf i} )  f_2     = 0 \quad \textrm{on }  \  \Omega_{\bf i}  ,$$
i.e., $f_1,f_2\in M_{v} (\Omega_{\bf i})$.
\item  From   the proof of Proposition \ref{G_vSR_v} one gets that  
 $  q \mapsto e^{<q, v> }f(q) $   
belongs to $\mathcal{SR}(\Omega)$ and  from  Representation Theorem for the slice regular functions one has that    
\begin{align*}   
& e^{<x+{\bf i}_q y, v> }f  (x+{\bf i}_q y) \\= & \frac {1}{2}\left\{    e^{<x+{\bf i}  y, v> }  (1-  {\bf i}_q {\bf i})   f  (x+{\bf i}  y) + 
e^{<x-{\bf i} y, v> } ( 1+  {\bf i}_q {\bf i})  f  (x-{\bf i}  y)   \right\} ,
\end{align*} 
or equivalently
\begin{align*}   
& f  (x+{\bf i}_q y)\\
 = & \frac {1}{2}\left\{    e^{< ({\bf i}-{\bf i}_q)  y, v> }  (1-  {\bf i}_q {\bf i})   f  (x+{\bf i}  y) + 
e^{<({\bf i}-{\bf i}_q) y, v> } ( 1+  {\bf i}_q {\bf i})  f  (x-{\bf i}  y)   \right\} .
\end{align*}

\end{enumerate}
\end{proof}

\begin{rem}
The previous proposition shows a deep relationship between the  $v$-slice regular-type functions
 with a kind of    complex generalize    analytic  functions  that are solutions of some  complex Vekua problems. 
\end{rem}

Let   ${\bf i},{\bf j}\in \mathbb{S}$ be  orthogonal to each other. Proposition \ref{SGV_Split_Repre} defines   the   operators 
$
     Q^v_{ {\bf i},{\bf j}} :    \mathcal{SR}_v(\Omega) \to  M_v(\Omega_{  {\bf i}})+ M_v(\Omega_{  {\bf i}}) {  {\bf j} }$ and 
		$ P^v_{  {\bf i},{\bf j} } :   M_v(\Omega_{  {\bf i}})+ M_v(\Omega_{  {\bf i}}){  {\bf j}} \to  \mathcal{SR}_v(\Omega)$ given  by 
$$ Q^v_{ {\bf i} , {\bf j} } [f ]= f\mid_{\Omega_{{\bf i}}} = f_1+f_2{\bf j},\quad \forall f\in\mathcal{SR}_v(\Omega),$$
 where  $f_1,f_2\in  M_v(\Omega_{\bf i}) $ and 
\begin{align*}
 &   P^v_{ {\bf i},{\bf j} }[g](x+ {\bf i}_q y)\\
=& \frac {1}{2}\left\{    e^{\langle    ({\bf i}-{\bf i}_q)  y, v\rangle    }  (1-  {\bf i}_q {\bf i})   f  (x+{\bf i}  y) + 
e^{\langle   ({\bf i}-{\bf i}_q) y, v\rangle    } ( 1+  {\bf i}_q {\bf i})  f  (x-{\bf i}  y)   \right\}
	,\end{align*} for all $x+{\bf i}_qy \in \Omega$ and for all  $g\in  M_v(\Omega_{  {\bf i} })+ M_v(\Omega_{  {\bf i} }){  {\bf j} }$.
		  			 What is more, 
 \begin{align*} P^v_{  {\bf i} ,  {\bf j}  }\circ Q^v_{ {\bf i} ,{\bf j}  }= \mathcal I_{\mathcal{SR}_v(\Omega)} \quad \textrm{and} \quad    Q^v_{  {\bf i} ,  
{\bf j} }\circ P^v_{ {\bf i} ,  {\bf j}   }= \mathcal I_{ M_v(\Omega_{ {\bf i} })+ M_v(\Omega_{ {\bf i} }){ {\bf j} } },
\end{align*} 
where $\mathcal I_{\mathcal{SR}_v(\Omega)}$  and $\mathcal I_{ M_v(\Omega_{ {\bf i} })+ M_v(\Omega_{ {\bf i} }){  {\bf j}  } }$ are  the identity
 operators in  $\mathcal{SR}_v(\Omega)$ and    in  ${ M_v(\Omega_{ {\bf i} })+ M_v(\Omega_{ {\bf i} }){ {\bf j} } }$, respectively.

\begin{prop}\label{Pro_Series}
If   $f \in\mathcal {SR}_v(\mathbb B^4(0,1))$ then   there exists a sequence of quaternions $(a_n)_{n=0}^{\infty} $ such that  
  $$f (q)=   \sum_{n=0}^{\infty} \sum_{k=0}^n      q^{n-k} (    v \bar q +  q\bar v  )^k \frac{(-1)^{k} 
a_{n-k}
 }{2^k k!} , \quad \forall   q\in\mathbb B^4(0,1). $$ 
Particularly, 
$$  
f(q)= \sum_{n=0}^{\infty}   q^n\sum_{k=0}^n         \frac{(-1)^{k} c_{v,k}(q) a_{n-k} }{2^k k!},  \quad \forall q\in \mathbb B^4(0,1)\setminus \{0\}.$$
where 
 $c_{v,k}(q)= (   \dfrac{\bar q}{\|q\|} v \dfrac{\bar q}{\|q\|}  +   \bar v  )^k$. 
\end{prop}
\begin{proof}
As  $f \in\mathcal {SR}_v(\mathbb B^4(0,1))$ from Proposition \ref{G_vSR_v} and equation  \eqref{G-Exp} one sees that  
 the mapping $q\mapsto  e^{\langle    q,v\rangle    } f(q) $  belongs to $\mathcal{SR}(\mathbb B^4(0,1))$ then 
  there exists a sequence of quaternions $(a_n)_{n=0}^{\infty} $ 
	such that 
$$\displaystyle 
f(q)=  e^{-\langle    q,v\rangle    }\sum_{n=0}^{\infty}q^n a_n = \sum_{n=0}^{\infty}\frac{(-1)^n}{n!}  \langle    q,v\rangle^n \sum_{n=0}^{\infty} q^n       a_n  ,$$
or equivalently, 
$$\displaystyle 
f(q)= \sum_{n=0}^{\infty} \sum_{k=0}^n      q^{n-k} (    v \bar q +  q\bar v  )^k \frac{(-1)^{k} a_{n-k} }{2^k k!}.$$
Particularly, if      $q\neq 0$ then 
$$\displaystyle 
f(q)= \sum_{n=0}^{\infty}   q^n\sum_{k=0}^n         \frac{(-1)^{k} c_{v,k}(q) a_{n-k} }{2^k k!}, $$
where 
$$c_{v,k}(q)= (   \frac{\bar q}{\|q\|} v \frac{\bar q}{\|q\|}  +   \bar v  )^k.$$

\end{proof}

\begin{rem}Using the $*$ product of slice regular functions we can establish the a multiplication in $\mathcal{SR}_v(\mathbb B^4(0,1)) $ as follows: 
\begin{align*}
  f*_v g(q)=  &  \sum_{n=0}^{\infty} \sum_{k=0}^n      q^{n-k} (    v \bar q +  q\bar v  )^k \frac{(-1)^{k}}{2^k k!} 
\displaystyle \sum_{m=0}^{n-k} a_m b_{n-k-m}
  ,\end{align*} 
 for all $q\in\mathbb B^4(0,1)$, where   $f,g \in \mathcal{SR}_v(\mathbb B^4(0,1)) $  and  
\begin{align*} 
f(q)=& \sum_{n=0}^{\infty} \sum_{k=0}^n      q^{n-k} (    v \bar q +  q\bar v  )^k \frac{(-1)^{k} a_{n-k} }{2^k k!},\\
g(q)=& \sum_{n=0}^{\infty} \sum_{k=0}^n      q^{n-k} (    v \bar q +  q\bar v  )^k \frac{(-1)^{k} b_{n-k} }{2^k k!},
\end{align*}
 for all $q\in\mathbb B^4(0,1)$. Note that 
$ f*_0 g =  f*  g$ for   $f,g \in \mathcal{SR}_0(\mathbb B^4(0,1)) $
\end{rem}

\begin{prop} ( {Cauchy's}  Integral Formula for $\mathcal {SR}_v(\Omega)$).
Let  $\Omega\subset \mathbb H$ be  an axially symmetric s-domain.   Given $f\in \mathcal SR_v(\Omega)$,  $a\in\mathbb R$ and $r>   0$ 
 such that  $q\in \overline{\Delta_q(a,r)} \subset  \Omega$ then
	   		\begin{align*}    
f(q) = 
\frac{1}{2\pi }\int_{\partial\Delta_q(a,r) }e^{\langle    \zeta-q,v\rangle     }  (\zeta- q)^{-1} d\zeta_{{\bf i}_q}   f (\zeta), \end{align*}
	where  $d\zeta_{{\bf i}_q} := - d\zeta {{\bf i}_q}$.
 
\end{prop}
\begin{proof} 
As the mapping   $q\mapsto e^{\langle    q,v\rangle    }f(q)$  belongs to $\mathcal SR(\Omega)$  then \eqref{CauchyForm_SR}  implies 
\begin{align*}    
e^{\langle    q,v\rangle     }f(q) = 
\frac{1}{2\pi }\int_{\partial\Delta_q(a,r) } (\zeta- q)^{-1} d\zeta_{{\bf i}_q} e^{\langle    \zeta,v\rangle    } f (\zeta), \end{align*}
or equivalently 
	\begin{align*}    
f(q) = 
\frac{1}{2\pi }\int_{\partial\Delta_q(a,r) }e^{\langle    \zeta-q,v\rangle     }  (\zeta- q)^{-1} d\zeta_{{\bf i}_q}   f (\zeta), \end{align*}
	where  $d\zeta_{{\bf i}_q} := - d\zeta {{\bf i}_q}$.
\end{proof}

\begin{cor} (Cauchy's  Integral Theorem for $\mathcal {SR}_v(\Omega)$).
   If $f\in \mathcal SR_v(\Omega)$ then  
\begin{align*} \int_{\Gamma} f e^{\langle    \zeta, v \rangle    }d\zeta_{\bf i} =0, \end{align*}
 for all ${\bf i}\in \mathbb S^2$ and for any closed,    { homotopic
to a point and } piecewise $C^1$       curve   $\Gamma\subset \Omega_{\bf i}$.
\end{cor}
\begin{proof}
 	     As  $q \mapsto e^{\langle    q,v\rangle    }f(q)$ belongs to $\mathcal SR(\Omega)$ from   \eqref{CauchyTheo_SR}   
				one has 
\begin{align*} \int_{\Gamma} e^{\langle    \zeta,v\rangle    }f(\zeta) d\zeta_{\bf i} =0, \end{align*}
 for all ${\bf i}\in \mathbb S^2$ and for any closed,    { homotopic
to a point and } piecewise $C^1$       curve   $\Gamma\subset \Omega_{\bf i}$.
\end{proof}

We finish this paper by showing 
 the    Identity Principle, Liouville's Theorem and  Morera's Theorem for $\mathcal{SR}_v(\Omega)$.

\begin{prop}
 Let  $\Omega\subset \mathbb H$ be an axially symmetric s-domain. 
\begin{enumerate}
 \item    Identity Principle for $\mathcal{SR}_v(\Omega)$.  
If $f\in \mathcal{SR}_v(\Omega)$ and there exists ${\bf i}\in\mathbb S^2$ such that  $Z_f \cap \Omega_{\bf i} =
 \{q \in \Omega \ \mid \  f (q) = 0\} \cap \Omega_{\bf i}$   has
an accumulation point then $f =0$  on $ \Omega$.

\item  Liouville's Theorem for $\mathcal{SR}_v(\Omega)$. 
If  $f\in\mathcal {SR}_v(\mathbb H)$ and  there exists 
$M>   0$ such that 
$$\|f(q)\|\leq  M e^{-\langle    q,v\rangle     }, \quad \forall q\in \mathbb H .$$
Then there exists   $k\in\mathbb H$ such that $f(q) = e^{-\langle    q,v\rangle    } k $ for all  $q\in \mathbb H$.

\item Morera's Theorem for $\mathcal{SR}_v(\Omega)$.   Suppose that  $h\in C(\Omega, \mathbb H)$ satisfies   
\begin{align*} \int_{\Gamma}h(\zeta) e^{-\langle    \zeta,v\rangle    }d\zeta_{\bf i} =0,
\end{align*}
for any closed,    {homotopic
to a point and } piecewise $C^1$ curve   $\Gamma\subset \Omega_{\bf i}$ 
  and for all ${\bf i}\in \mathbb S^2$.      Then   	$h\in \mathcal{SR}_v(\Omega)$.
\end{enumerate}
\end{prop}

\begin{proof} 
\begin{enumerate}

\item       
The zero set of the slice regular  function $q\mapsto e^{\langle    q, v\rangle     }f(q) $ has an accumulation point  and from 
  identity Principle for the  slice regular functions 
  one gets that $e^{\langle    q, v\rangle     }f(q) =0$  for all  $q\in \Omega$.
		Therefore $f(q) =0$  for all  $q\in \Omega$.

\item  The slice regular functions $q\mapsto e^{\langle    q,v\rangle    }f(q) $  is a bounded function on $\mathbb H$ and   Liouville's Theorem for slice regular functions implies  that $q\mapsto e^{\langle    q,v\rangle    }f(q)$ is a constant function.

\item  As  $h:\Omega\to \mathbb H$ satisfies   
\begin{align*} \int_{\Gamma}h e^{\langle    \zeta,v\rangle    } d\zeta =0,
\end{align*}
for any closed,    {homotopic
to a point and } piecewise $C^1$ curve   $\Gamma\subset \Omega_{\bf i}$ 
  and for all ${\bf i}\in \mathbb S^2$.     Then    
	  $q\mapsto  e^{\langle    q,v\rangle     h(q)} $ satisfies the  Morera's Theorem for slice regular functions  and  
		$ h\in   \in \mathcal{SR}_v(\Omega)$.

\end{enumerate}
\end{proof}

\end{document}